\documentclass[11pt,a4paper]{amsart}
\usepackage{amscd,amssymb,amsopn,amsmath,amsthm,graphics,amsfonts,enumerate,verbatim,calc}
\usepackage[dvips]{graphicx}
\usepackage{color}

\usepackage{amssymb,amsmath}

\headsep=1cm \numberwithin{equation}{section}
\newtheorem{theorem}{Theorem}[section]
\newtheorem{lemma}[theorem]{Lemma}
\newtheorem{proposition}[theorem]{Proposition}
\newtheorem{corollary}[theorem]{Corollary}

\newtheorem{Maintheorem}{Main Theorem}
\newtheorem{question}[theorem]{Question}

\theoremstyle{definition}
\newtheorem{definition}[theorem]{Definition}

\theoremstyle{remark}
\newtheorem{remark}[theorem]{Remark}
\newtheorem{example}[theorem]{Example}

\newtheorem{acknowledgement}{Acknowledgement}

\newcommand{\Ass}{\operatorname{Ass}}

\newcommand{\Spec}{\operatorname{Spec}}

\newcommand{\Ht}{\operatorname{ht}}

\newcommand{\gr}{\operatorname{gr}}

\newcommand{\Supp}{\operatorname{Supp}}

\newcommand{\depth}{\operatorname{depth}}

\newcommand{\coker}{\operatorname{coker}}

\newcommand{\Pic}{\operatorname{Pic}}

\newcommand{\Sp}{\operatorname{Sp}}

\newcommand{\Reg}{\operatorname{Reg}}

\newcommand{\Sing}{\operatorname{Sing}}
\newcommand{\WN}{\operatorname{WN}}

\newcommand{\fm}{\frak{m}}
\newcommand{\fp}{\frak{p}}
\newcommand{\fq}{\frak{q}}

\newcommand{\fn}{\frak{n}}

\begin{document}
\title[Weak normality and seminormality in the mixed characteristic case]
{Weak normality and seminormality in the mixed characteristic case}

\author[J. Horiuchi]{Jun Horiuchi}
\address{Department of Mathematics, Nippon Institute of Technology, Miyashiro, Saitama 345-8501, Japan}
\email{jhoriuchi.math@gmail.com}

\author[K. Shimomoto]{Kazuma Shimomoto}
\address{Department of Mathematics, College of Humanities and Sciences, Nihon University, Setagaya-ku, Tokyo 156-8550, Japan}
\email{shimomotokazuma@gmail.com}



%\subjclass{13}
%\subjclass[2000]{Primary 13-XX}
%\subjclass[2000]{Primary ; Secondary}
%\date{\today \, (\printtime)}
%\date{\today}

\begin{abstract}
In this article, we study certain properties of Noetherian rings with weak normality and seminormality in mixed characteristic. It is known that the two concepts can differ in the equal prime characteristic case, while they coincide by definition in the equal characteristic zero case. We exhibit some examples in the mixed characteristic case. We also establish the local Bertini theorem for weak normality in mixed characteristic under a certain condition.
\end{abstract}

\maketitle

\section{Introduction}

The purpose of this article is to fill in the gaps in the literature on weak normality and seminormality on commutative rings in the mixed characteristic case. Despite the fact that these notions are defined in a characteristic-free manner, it seems that there are not many articles that pay special attention to the problem of arithmetic variations of weak normality or seminormality in the geometric setting. For instance, we are motivated by the following question.

\begin{question}
\label{geometric}
Let $f:X \to \Spec(V)$ be a flat surjective scheme map of finite type, where $V$ is a Dedekind domain of mixed characteristic. Suppose that the generic fiber of $f$ has weakly normal (or seminormal) singularities. Then what can one say about the (general) closed fibers of $f$?
\end{question}

Conversely, one can ask about the singularities of the generic fiber of $f$ by requiring some closed fiber (or general closed fibers) of $f$ to possess weak normal (or seminormal) singularities. Let us emphasize that one can offer good answers to this converse question by relating it to the \textit{deformation property} of weakly normal (or seminormal) singularities. For the precise statement, we refer the reader to \cite{He08} for the deformation of seminormality, and to \cite[Proposition 4.11]{Mu20} for the deformation of weak normality. The authors think that the study of arithmetic variations of singularities is becoming of immediate importance, as the recent breakthrough paper \cite{MS19} has established a bridge connecting $F$-singularities and singularities appearing in the Minimal Model Program. Moreover, the singularities that are discussed in \cite{MS19} are not necessarily normal, and often fall into the types of singularities treated in the present article (a typical situation is that the generic fiber could be Du Bois, and some closed fiber could be $F$-injective. These singularities are not necessarily normal, but weakly normal).  Let us state the first main theorem; see Theorem \ref{weaknormalBertini}.

\begin{Maintheorem}
Let $(V,\pi,k)$ be an unramified complete discrete valuation ring of mixed characteristic with residue field of characteristic $p>0$. Suppose that $(R,\fm,k)$ is a $V$-flat complete local domain which induces an isomorphism on residue fields $k \cong V/(\pi) \cong R/\fm$ and that the following conditions hold:
\begin{enumerate}
\item
$R \to \overline{R}$ is unramified in codimension $1$, where $\overline{R}$ is the integral closure of $R$ in the field of fractions of $R$;

\item
$x_0,\ldots,x_d$ is a fixed set of minimal generators of $\fm$;

\item
the residue field $k$ is infinite.
\end{enumerate}

Then there exists a non-empty Zariski open subset $\mathcal{U} \subset \mathbb{P}^{d}(k)$ such that for a fixed $\alpha=(\alpha_0:\cdots:\alpha_d) \in \Sp^{-1}_V(\mathcal{U})$, the localization $(R/\mathbf{x}_{\widetilde{\alpha}}R)_\fp$ is a weakly normal local reduced ring of mixed characteristic for every $\fp \in \WN(R) \cap V(\mathbf{x}_{\widetilde{\alpha}}) \cap \Spec^\circ(R)$, where we put
$$
\mathbf{x}_{\widetilde{\alpha}}=\sum_{i=0}^d \widetilde{\alpha}_i x_i.
$$
\end{Maintheorem}

The second result is the following; see Example \ref{notweaklynormal} and Example
\ref{weaklynormalR_1} respectively.

\begin{Maintheorem}
The following assertions hold:

\begin{enumerate}
\item
There is a local Noetherian domain $(R,\fm)$ of mixed characteristic such that $R$ is seminormal, but not weakly normal.

\item
There is a local Noetherian domain $(R,\fm)$ of mixed characteristic such that $R$ is weakly normal with Serre's $(R_1)$-condition, but not normal.
\end{enumerate}
\end{Maintheorem}

Although the main results in this paper do not provide direct answers to Question \ref{geometric}, the authors hope that they will motivate the readers to develop methods for studying the arithmetic properties of singularities.

\section{Weak normality and seminormality}

In this section, let us recall briefly definitions of weak normalization and seminormalization. Throughout this article, rings are commutative with unity. The theory of seminormal rings arises from the following question: Let $A$ be a commutative ring. When are the Picard group of $A$ and the Picard group of the polynomial ring over $A$ equal to each other? Traverso \cite{T70} and Hamann \cite{H75} have answered this question as follows.

\begin{theorem}
\label{Traverso}
Let $A$ be a reduced Noetherian commutative ring with total quotient ring $Q(A)$. Then the following conditions are equivalent.

\begin{enumerate}
\item
$\Pic(A)=\Pic(A[X])$.

\item
$\Pic(A)=\Pic(A[X_1,\ldots,X_n])$ for all $n\in \mathbb{N}$.

\item
$A$ is seminormal in $Q(A)$.

\item
If $a\in Q(A)$ and $a^2, a^3 \in A$, then $a\in A$.
\end{enumerate}
\end{theorem}

In \cite{T70}, it was assumed that $A$ had finite normalization in $Q(A)$, but this assumption was later eliminated by Gilmer-Heitmann in \cite{GH80}. On the other hand, the study of weakly normal rings has its roots in investigating weakly normal varieties over the complex numbers, due to Andreotti-Bombieri \cite{AB69}. Let $A \subset B$ be an integral extension of commutative rings and consider the subrings of $B$:
$$
A_B^+=\Big\{b \in B~\Big|~\frac{b}{1} \in A_{\fp}+J(B_{\fp})~\mbox{for all}~\fp \in \Spec(A)\Big\}
$$
and
$$
A_B^*=\Big\{b \in B~\Big|~\mbox{for all}~\fp \in \Spec(A),~\mbox{there exists}~n \in \mathbb{N}~\mbox{such that}~{\Big(\frac{b}{1}\Big)}^{p^n} \in A_{\fp}+J(B_{\fp})\Big\},
$$
where the symbol $J(A)$ denotes the Jacobson radical of the ring $A$ and $p$ is the characteristic exponent of the field $k(\fp)=A_{\fp} /\fp A_{\fp}$. We say that $A_{B}^+$ is the \textit{seminormalization} of $A$ in $B$, and $A_B^*$ is the \textit{weak normalization} of $A$ in $B$. Recalling that the characteristic exponent of an equal characteristic zero field is $1$, obviously $A_B^+ \subseteq A_B^*$ with equality in the equal characteristic zero case. Let $Q(A)$ be the total quotient ring of $A$. In the case that $B$ is the normalization of $A$ in $Q(A)$, we write $A^+$ instead of $A_B^+$ (resp. $A^*$ instead of $A_B^*$). If $A^+=A$, then we say that $A$ is \textit{seminormal}, and if $A^*=A$, then we say that $A$ is \textit{weakly normal}. From the definition, we have the implications: normal $\Rightarrow$ weakly normal $\Rightarrow$ seminormal.

Swan redefined the definition of seminormal rings modifying the characterization by square-cubic Theorem \ref{Traverso}(4) without mentioning an extension of rings. It is common to adopt the generalized definitions due to Swan \cite{S80} and Yanagihara \cite{Y85}. We will limit our attention to only Noetherian rings, although the definition makes sense over general commutative rings.

\begin{definition}
Let $A$ be a commutative Noetherian ring. 
\begin{enumerate}
\item[\rm{(i)}]
For any elements $b,c \in A$ with $b^3=c^2$, there exists an element $a \in A$ satisfying $b=a^2, c=a^3$.

\item[\rm{(ii)}]
For any elements $b,c,e \in A$ and any non-zero divisor $d \in A$ with $c^p=bd^p$ and $pc=de$ for some prime integer $p>0$, there is an element $a \in A$ with $b=a^p$ and $e=pa$.
\end{enumerate}

We call a ring $A$ which satisfies the condition $\rm{(i)}$ \textit{seminormal in the sense of Swan}, and which satisfies both conditions $\rm{(i)}$ and $\rm{(ii)}$ \textit{weakly normal in the sense of Yanagihara}.
\end{definition}

\begin{remark}
If the ring is seminormal in the sense of Swan, then it is seminormal in the original sense and the reverse implication holds when the ring is reduced. In the same way, if the ring is weakly normal in the sense of Yanagihara, then it is weakly normal in the original sense and the reverse implication holds when the ring is reduced. Notice that the condition $\rm{(i)}$ implies that the ring is necessarily reduced. For the proof of these facts, we refer the reader to \cite{S80} and \cite{Y85}.
\end{remark}

We close this section by introducing one characterization of weakly normal rings. We use this characterization in the proof of the main theorem; see \cite[Theorem 1.6]{M80} for the proof.

\begin{theorem}[Manaresi]
\label{characterization}
Let $R\subseteq S$ be an integral extension of commutative Noetherian rings. Then $R$ is weakly normal in $S$ if and only if the sequence of $R$-modules:
$$
R \to S \overset{f}{\underset{g}\rightrightarrows} (S \otimes_{R}S)_{\rm{red}}
$$
is exact, where $f(b)=b \otimes 1 \pmod{\sqrt{0}}$ and $g(b)=1 \otimes b \pmod{\sqrt{0}}$. That is, $R$ is isomorphic to the equalizer of $S \overset{f}{\underset{g}\rightrightarrows} (S  {\otimes}_{R}S)_{\rm{red}}$.
\end{theorem}

Let $R \to S$ be a ring homomorphism. Then we say that $R \to S$ is \textit{unramified in codimension $1$} if the localization map $R_\fp \to S_\fp$ is unramified for all height-$1$ primes $\fp \in \Spec(R)$.

\section{An example of a weakly normal ring in mixed characteristic}

We present an example of a local ring in mixed characteristic that is seminormal, but not weakly normal. Examples of this type do not seem to abound in the existing literature.

\begin{example}
\label{notweaklynormal}
Let $\mathbb{Z}_2$ be the ring of $2$-adic integers, and let $\mathbb{Z}_2[[X,Y]]$ be the ring of formal power series ring in indeterminates $X$ and $Y$. We put
$$
A:=\mathbb{Z}_2[[X,Y]]/(Y^2-4X).
$$ 
Then $A$ is a $2$-dimensional local domain of mixed characteristic. We write $x,y$ for the images of $X,Y$, respectively. Consider the natural injection into the field of fractions $Q(A)=\mathbb{Q}_2((y))$
$$
A=\mathbb{Z}_2[[X,Y]]/(Y^2-4X) \hookrightarrow Q(A)=\mathbb{Q}_2((y)).
$$
We show that $A$ is seminormal, but not weakly normal. We check that $A$ is seminormal. 
In view of \cite[Corollary 2.7]{GT80}, it suffices to check that $A$ satisfies $(S_2)$, which it clearly does, and that $A_{\fp}$ is seminormal for each height-$1$ prime $\fp$. If $2 \notin \fp$, then $A_{\fp}$ is regular and there nothing to prove. So let us assume $2 \in \fp$. Then we get $\fp=(2,y)$ which is the only height-one prime containing $2$. After completing, we get
$$
\widehat{A_\fp} \cong V[[T]]/(T^2-4u),
$$
where $V$ is an unramified complete DVR and $u \in V[[T]]^{\times}$ is a unit, which is obtained as follows. The residue field of $\widehat{A_\fp}$ is isomorphic to $\mathbb{F}_2((u))$. Here, $u$ is the image of $X$ under the map $A \to \widehat{A_\fp}$. Let $\fm$ be the maximal ideal of $\widehat{A_\fp}$. Then $e(\widehat{A_\fp})=2=\dim_{\mathbb{F}_2((u))} (\fm/\fm^2)$. Moreover,
$$
\gr_{\fm}(\widehat{A_\fp}) \cong \mathbb{F}_2((u))[S,T]/(S^2-T^2u)
$$
and $u \in \mathbb{F}_2((u))$ does not admit a square root and therefore, 
$\gr_{\fm}(\widehat{A_\fp})$ is reduced. By Davis' result \cite{D78}, we see that $\widehat{A_\fp}$ is seminormal and so is $A_\fp$ by \cite[Corollary 5.3]{GT80}.

Next we check this ring is not weakly normal. Let us take $\frac{y}{2} \in Q(A)$. Then it is easy to check, $2 (\frac{y}{2})=y \in A$, $(\frac{y}{2})^2=\frac{y^2}{4}=\frac{4x}{4}=x \in A$, and  $\frac{y}{2}$ is not contained in $A$. Therefore, our ring is not weakly normal. 
\end{example}

We want to include a second proof of Example \ref{notweaklynormal} due to Karl Schwede which uses the method of pull-backs. The authors would like to thank him for permitting us to include this result. Before stating the result, let us define one term. Let $A$ be a ring of characteristic $p>0$, and let $B$ be a subring of $A$. We say that the ring extension $B\subset A$ is \textit{generically purely inseparable} if, for any $a\in A$, there exists an integer $e\geq 0$ such that $a^{p^e}\in B$.

\begin{theorem}[K. Schwede]
\label{pull-back}
Let $R$ be a Noetherian normal local domain of mixed characteristic with residue field of characteristic $p>0$ and let $I\subset R$ be any proper ideal such that $R/I$ is reduced. Assume that $B$ is a seminormal Noetherian subring of $R/I$ such that $B \to R/I$ is module-finite and generically purely inseparable. Then the pull-back of the diagram $R\to R/I \leftarrow B$ is a Noetherian seminormal local domain of mixed characteristic with residue field of characteristic $p>0$, which is not weakly normal. 
\end{theorem}

\begin{proof}
Let $R'$ be the pull-back of the diagram $R\to R/I \leftarrow B$. First off, the seminormality of $R'$ follows from \cite[Lemma 2.23]{V11}. We put $X=\Spec(R)$, $Y=\Spec(B)$ and $Z=\Spec(R/I)$. Since $B \to R/I$ is generically purely inseparable, the induced map $Z \to Y$ is a homeomorphism with purely inseparable (or trivial) residue field extensions. From now on, we will show that $R'\subset R$ is a weakly subintegral and birational extension. Since $\Spec(R') =X\cup_Z Y$ (see \cite[Theorem 3.4]{Sch05}), we have an isomorphism:
$$
\Spec(R') \setminus V(R' \cap I) \cong X \setminus V(I),
$$
where $V(I)=Y$. Thus, $R' \to R$ is a birational extension. On the other hand, notice that $V(R' \cap I) \to V(I)$ coincides with $Z \to Y$, which was already seen to be weakly subintegral. Hence we have a (non-trivial) weakly subintegral birational extension $R'\subset R$. In particular, $R'$ is not weakly normal.
It is not hard to check that $R'$ is a Noetherian local domain with mixed characteristic and we leave it as an exercise.
\end{proof}

\begin{proof}[Proof of Example \ref{notweaklynormal}]
With the same notation as in Theorem \ref{pull-back}, we put $R=\mathbb{Z}_2[[T]]$ and set $I$ to be the ideal generated by $2$. Then we have $R/I \cong \mathbb{F}_2[[T]]$. Let $B=\mathbb{F}_2[[T^2]]$ be a subring of $\mathbb{F}_2[[T]]$. Notice that
$$
B= \mathbb{F}_2[[T^2]]\to \mathbb{F}_2[[T]]\cong R/I
$$
satisfies the hypothesis of Theorem \ref{pull-back} (c.f. \cite[Example 2.13]{V11}). By a simple calculation, we see that $2T$, $T^2$ and $\mathbb{Z}_2$ generate the pull-back as a ring. We set $Y=2T$, $X=T^2$ to get the relation $Y^2 -4X=0$. Then the pull-back of the diagram $R\to R/I \leftarrow B$ is the same ring as $A$.
\end{proof}

\begin{remark}
For any prime $p>0$, setting $B=\mathbb{F}_p[[T^p]]$ as a subring of $R=\mathbb{F}_p[[T]]$, we can modify Example \ref{notweaklynormal} to get the similar result in mixed characteristic $p>0$.
\end{remark}

\section{Local Bertini theorem for weak normality in mixed characteristic}

In this section, we discuss the local Bertini theorem for weak normality in mixed characteristic. The second-named author and Ochiai proved the local Bertini theorem for normality in the mixed characteristic case; see \cite[Theorem 4.4]{OS15}. Cumino, Greco and Manaresi studied the Bertini theorem for weak normality in characteristic zero in \cite{CGM83}. To state our theorem, we need some preparation and we begin with the definition of the specialization map to formulate the local Bertini theorem in mixed characteristic.

\begin{definition}[Specialization map]
\label{spmap}
Let $(V,\pi,k)$ be a discrete valuation ring. Recall the construction of the \textit{specialization map} $\Sp_V:\mathbb{P}^n(V) \to \mathbb{P}^n(k)$. Let us pick a point $\alpha=(\alpha_0:\dots:\alpha_n) \in \mathbb{P}^n(V)$ with its lift $\widetilde{\alpha}=(\widetilde{\alpha}_0,\ldots,\widetilde{\alpha}_n) \in V^{n+1} \setminus \{0,\ldots,0\}$. Then we define
$$
\Sp_V(\alpha):=(\overline{\alpha}_0:\cdots:\overline{\alpha}_n) \in \mathbb{P}^n(k),
$$
where we put $\overline{\alpha}_i:=\widetilde{\alpha}_i \pmod {\pi_V}$. 
\end{definition}

Every point of $\mathbb{P}^n(V)$ is normalized and this map is independent of the lift of $\alpha=(\alpha_0:\cdots:\alpha_n)$. Therefore, the specialization map is well defined. Let $(R,\fm,k)$ be a Noetherian local $V$-algebra and pick a system of elements $x_0,\ldots,x_n$ from the maximal ideal $\fm$ and choose a point $\alpha=(\alpha_0:\cdots:\alpha_n) \in \mathbb{P}^n(V)$. Let us put
$$
\mathbf{x}_{\widetilde{\alpha}}:=\sum_{i=0}^n \widetilde{\alpha}_i x_i,
$$
where $\widetilde{\alpha}=(\widetilde{\alpha}_0,\ldots,\widetilde{\alpha}_n) \in V^{n+1} \setminus \{0,\ldots,0\}$ is a lift of $\alpha=(\alpha_0:\cdots:\alpha_n) \in \mathbb{P}^n(V)$ through the quotient map $V^{n+1} \setminus \{0,\ldots,0\} \to \mathbb{P}^n(V)$. The principal ideal $\mathbf{x}_{\widetilde{\alpha}}R$ does not depend on the lift of $\alpha \in \mathbb{P}^n(V)$.

For an ideal $I\subseteq R$ of a Noetherian ring $R$, we denote by $V(I)$ the set of points of $\Spec(R)$ which contain $I$. We denote by $\Reg(R)$ the \textit{regular locus} of $\Spec (R)$ and by $\Sing(R)$ the \textit{singular locus} of $\Spec(R)$. Denote by $\Spec^\circ(R)$ the complement of the set of all maximal ideals in $\Spec(R)$. Finally, denote by $\WN(R)$ the set of $\fp \in \Spec R$ such that $R_\fp$ is weakly normal. We denote the \textit{$n$-th symbolic power ideal} of $\fp$ by $\fp^{(n)} = \fp^{n} R_{\fp} \cap R$. We need a generalization of \cite[Theorem 4.3]{OS15}:

\begin{theorem}
\label{local-Bertini-Normal}
Suppose that $(R,\fm,k)$ is a complete local domain of mixed characteristic with residue field of characteristic $p>0$ and that the following conditions hold: 
\begin{enumerate}
\item
$V \to R$ is a coefficient ring map, where $(V,\pi,k)$ is an unramified complete discrete valuation ring which induces $V/(\pi) \cong R/\fm$;

\item
$x_0,x_1,\ldots,x_d$ is a set of elements of $\fm$;

\item
the residue field $k$ is infinite.
\end{enumerate}
Consider the natural map of $R$-modules:
$$
\phi:\bigoplus_{i=0}^n Rdx_i \to \widehat{\Omega}_{R/V}
$$
and let $W$ be the subset of the punctured spectrum $\Spec^\circ(R)$ consisting of prime ideals $\fp$ for which the localization map $\phi_{\fp}$ is surjective. Then $W$ is open in $\Spec(R)$ and there exists a non-empty Zariski open subset $\mathcal{U}' \subseteq \mathbb{P}^{d}(k)$ such that
$$
\mathbf{x}_{\widetilde{\alpha}}:=\sum_{i=0}^d \widetilde{\alpha}_i x_i \notin \fp^{(2)}
$$
for every $\fp \in W$ and every $\alpha=(\alpha_0:\cdots:\alpha_d) \in \Sp^{-1}_V(\mathcal{U}') \subseteq \mathbb{P}^{d}(V)$.
\end{theorem}

\begin{proof}
Since $\phi$ is a homomorphism between finitely generated $R$-modules, we see that the support of $\coker(\phi)$ is a closed subset of $\Spec(R)$. Since $\Spec^\circ(R)$ is open in $\Spec(R)$, it follows that $W$ is open in $\Spec(R)$. We may assume that $W$ is non-empty without loss of generality. Fix $\fp \in W$. Then we have $\mu_\fp(\widehat{\Omega}_{R/V}) \ge \dim(R/\fp)-1$ by \cite[Lemma 2.6]{Fl77}. Since $R$ is a catenary local domain and $W \subset \Spec^\circ(R)$ is a non-empty open subset, we have that $\dim(R/\fp)-1=\dim\big(V(\fp) \cap W\big)$: Indeed, there is a nonzero ideal $I \subset R$ such that $W=\Spec^\circ(R) \setminus V(I)$. Then any prime ideal $\fq \subset \Spec(R)$ that is maximal such that $\fp \subset \fq$ and $I \not \subset \fq$ satisfies $\Ht(\fq)=\dim R-1$ and $\fq \in W$. So we have $\dim\big(V(\fp) \cap W\big)=\Ht(\fq)-\Ht(\fp)=\dim(R/\fp)-1$ and
\begin{equation}
\label{minimal}
\mu_\fp(\widehat{\Omega}_{R/V}) \ge \dim\big(V(\fp) \cap W\big)~\mbox{for every}~\fp \in W.
\end{equation}
Using $(\ref{minimal})$, the rest of the proof proceeds by modifying slightly the proof of \cite[Theorem 4.3]{OS15}, and we indicate only main ideas for the sake of readability.

Let $\phi:R \to R[X_0,\ldots,X_d]$ be the natural map with the associated map $\phi^*:\Spec (R[X_0,\ldots,X_d]) \to \Spec(R)$. By $(\ref{minimal})$ together with \cite[Lemma 3.4]{OS15}, there is an ideal $(F_1,\ldots,F_r) \subset R[X_0,\ldots,X_d]$ such that
\begin{equation}
\label{dimensionformula}
\dim \big(V(F_1,\ldots,F_r) \cap (\phi^*)^{-1}(W)\big) \le d+1
\end{equation}
and
$$
\sum_{n=0}^d dx_n \otimes X_n \in \widehat{\Omega}_{R/V} \otimes_R R[X_0,\ldots,X_d]
$$
is a basic element at every point of $U(F_1,\ldots,F_r) \cap (\phi^*)^{-1}(W)$. Notice that $W \subset U(\fm)=\Spec^\circ(R)$. Let $T$ be the localization of $R[X_0,\ldots,X_d]$ with respect to $\fm R[X_0,\ldots,X_d]$. We have the localization map $\psi:R[X_0,\ldots,X_d] \to T$ and the flat local map $\psi \circ \phi:R \to T$. If we can show that the subset $V(F_1,\ldots,F_r) \cap (\phi^* \circ \psi^*)^{-1}(W) \subset \Spec(T)$ is finite, then the remainder of the proof follows exactly as the proof of \cite[Theorem 4.3]{OS15} by considering the prime ideals contained in $W$. To deduce the finiteness of $V(F_1,\ldots,F_r) \cap (\phi^* \circ \psi^*)^{-1}(W)$, we can mimic the proof of \cite[Lemma 3.5]{OS15} by using $(\ref{dimensionformula})$. We may assume that $V(F_1,\ldots,F_r) \cap (\phi^* \circ \psi^*)^{-1}(W)$ is not empty. In particular, this implies that there is a $\fq_0 \in V(F_1,\ldots,F_r) \cap (\phi^*)^{-1}(W)$ such that $\fq_0 \subset \fm R[X_0,\ldots,X_d]$. It suffices to show that the set of prime ideals $\fq \in V(F_1,\ldots,F_r) \cap (\phi^*)^{-1}(W)$ such that
\begin{equation}
\label{setprime}
\fq \subset \fm R[X_0,\ldots,X_d]
\end{equation}
is finite. To get a contradiction, assume that there is another $\fq_1 \in V(F_1,\ldots,F_r) \cap (\phi^*)^{-1}(W)$ such that $\fq_1 \subset \fm R[X_0,\ldots,X_d]$ and $\fq_0  \subsetneq \fq_1$. Then $\fq_1 \ne \fm R[X_0,\ldots,X_d]$. Consider the surjection of catenary domains: $R[X_0,\ldots,X_d]/\fq_1 \twoheadrightarrow k[X_0,\ldots,X_d]$. Since the kernel of this surjection is not trivial and $\dim\big(k[X_0,\ldots,X_d]\big)=d+1$, it follows that
\begin{equation}
\label{dimensioncatenary}
\dim\big(R[X_0,\ldots,X_d]/\fq_1\big) \ge d+2.
\end{equation}

There is a subideal $I \subset \fm$ such that $W=U(I)$ and the condition $\fq_1 \in (\phi^*)^{-1}(W)$ implies that $IR[X_0,\ldots,X_d] \not \subset \fq_1$ and thus, there is an element $b \in I$ satisfying $b \notin \fq_1$. Using $(\ref{dimensioncatenary})$, we can construct a chain of primes ideals of $R[X_0,\ldots,X_d]$:
\begin{equation}
\label{saturatedchainprime}
\fq_0 \subsetneq \fq_1 \subsetneq \fq_2 \subsetneq \cdots \subsetneq \fq_{d+2}~\mbox{such that}~b \notin \fq_{d+2}.
\end{equation}
As $b \notin \fq_{d+2}$, we have $\fq_i \in V(F_1,\ldots,F_r) \cap (\phi^*)^{-1}(W)$ for $i=0,\ldots,d+2$. However, $(\ref{saturatedchainprime})$ is of length $d+2$, which clearly violates $(\ref{dimensionformula})$. Thus, no such $\fq_1$ exists and we conclude that any prime ideal in $V(F_1,\ldots,F_r) \cap (\phi^*)^{-1}(W)$ that satisfies the condition $(\ref{setprime})$ is minimal over the ideal $(F_1,\ldots,F_r)$. Therefore, $V(F_1,\ldots,F_r) \cap (\phi^* \circ \psi^*)^{-1}(W)$ is finite.
\end{proof}

\begin{remark}
Let $R$ be a Noetherian ring and let $\fp \in \Spec(R)$ such that $R_{\fp}$ is regular. Take an element $x\in \fp$. If $x\notin \fp^{(2)}$, then we can show that the localization of $R/xR$ at $\fp$ is regular. Since $R_\fp$ is regular, $R_{\fp}/xR_{\fp}$ is regular too. Thus Theorem \ref{local-Bertini-Normal} yields the inclusion $\Reg (R) \cap V(\mathbf{x}_{\widetilde{\alpha}}) \subseteq \Reg(R/\mathbf{x}_{\widetilde{\alpha}}R)$ holds true for $\alpha=(\alpha_0:\cdots:\alpha_d) \in \Sp^{-1}_V(\mathcal{U}')$.
\end{remark}

\begin{lemma}
\label{coefficientring}
Let $R$ be a complete local domain with coefficient ring map $A \to R$, where $A$ is an unramified complete discrete valuation ring. Assume that $R \to S$ is a module-finite extension such that $S$ is normal. Then there is a coefficient ring map $B \to S$ for which there is a commutative square:
$$
\begin{CD}
R @>>> S \\
@AAA @AAA \\
A @>>> B \\
\end{CD}
$$
\end{lemma}

\begin{proof}
The proof is obtained by making a slight modification of the proof of \cite[Theorem 29.1]{M86} as follows: Instead of taking $L$ to be the algebraic closure of $A$ therein, one takes $L$ to be the field of fractions of $S$ and consider the condition $(*)$ as in \cite[Theorem 29.1]{M86}. By Zorn's lemma, one can find a valuation ring $B$ as demanded. The sought ring $B$ will be contained in $S$ as $S$ is assumed to be normal.
\end{proof}

\begin{proposition}
\label{semilocalnormalBertini}
Let $(V,\pi,k)$ be an unramified discrete valuation ring of mixed characteristic with infinite residue field and let $(R,\fm,k)$ be a $V$-flat excellent local domain such that $V \to R$ induces an isomorphism $V/(\pi) \cong R/\fm \cong k$. Assume that $x_0,\ldots,x_d$ is a minimal system of generators of $\fm$ and 
the map $R \to \overline{R}$ is unramified in codimension $1$, where $\overline{R}$ is the integral closure of $R$ in the field of fractions of $R$.

Then there is a non-empty Zariski open subset $\mathcal{U} \subset \mathbb{P}^d(k)$ for which $(\overline{R}/\mathbf{x}_{\widetilde{\alpha}}\overline{R})_\fp$ is a semi-local reduced normal ring of mixed characteristic for every $\alpha=(\alpha_0:\cdots:\alpha_d) \in \Sp^{-1}_V(\mathcal{U})$ and every $\fp \in \Spec^\circ(R) \cap V(\mathbf{x}_{\widetilde{\alpha}})$. If moreover $\depth_\fm \overline{R} \ge 3$, then $\overline{R}/\mathbf{x}_{\widetilde{\alpha}}\overline{R}$ is a semi-local reduced normal ring.
\end{proposition}

\begin{proof}
Let $\widehat{\overline{R}}$ be the $\fm$-adic completion of $\overline{R}$. Since $R$ is excellent, it follows that $\widehat{\overline{R}} \cong \widehat{R} \otimes_R \overline{R}$. Since $\widehat{\overline{R}}$ is a semi-local normal ring, we get $\widehat{\overline{R}} \cong \bigoplus_{i=1}^m S_i$, where $(S_i,\fm_i,K_i)$ is a complete normal local domain with residue field $K_i$. (If moreover 
$\depth_\fm \overline{R} \ge 3$, then $\depth_{\fm_i} S_i \ge 3$.) Moreover, the natural mapping:
\begin{equation}
\label{faithfulflat}
\overline{R} \to \widehat{\overline{R}}=\bigoplus_{i=1}^m S_i
\end{equation}
is faithfully flat. We note that the $\pi$-adic completion of $V$ which is $A:=\widehat{V}$ gives a coefficient ring for $\widehat{R}$. Now we prove the following claim:
\begin{enumerate}
\item[$(\#)$:]
Let $\Sp_{A}:\mathbb{P}^d(A) \to \mathbb{P}^n(k)$ be the specialization map.
 For each $i$, there is a non-empty open subset $\mathcal{V}_i \subset \mathbb{P}^d(k)$ such that $S_i/\mathbf{x}_{\widetilde{\alpha}}S_i$ is a normal domain of mixed characteristic, where
$$
\mathbf{x}_{\widetilde{\alpha}}:=\sum_{i=0}^d \widetilde{\alpha}_i x_i
$$
for every $\alpha=(\alpha_0:\cdots:\alpha_d) \in \Sp^{-1}_A(\mathcal{V}_i)$.
\end{enumerate}

For brevity of notation, we write $S=S_i$ and $K=K_i$ and let $\fn$ be its maximal ideal. Before starting the proof, we emphasize that the image of $x_0,\ldots,x_d$ in $S$ generates an $\fn$-primary ideal which may not be $\fn$ itself. To show the normality of $(S/\mathbf{x}_{\widetilde{\alpha}}S)_\fp$, it suffices to check Serre's $(R_1)$ and $(S_2)$-conditions. So let $\fp \in \Spec^\circ(S) \cap V(\mathbf{x}_{\widetilde{\alpha}})$. If $\Ht(\fp) > 2$, then we need to show that $\depth (S/\mathbf{x}_{\widetilde{\alpha}}S)_\fp \ge 2$. If $\Ht(\fp)=2$, then we need to show that $(S/\mathbf{x}_{\widetilde{\alpha}}S)_\fp$ is a discrete valuation ring.

First, we deal with the case $\Ht(\fp)=2$. The point in this case is to modify the proof of Theorem \ref{local-Bertini-Normal} as needed. We consider the exact sequence of completed modules of K\"ahler differentials applied to $A \to \widehat{R} \to S$:
\begin{equation}
\label{KahlerExt}
\widehat{\Omega}_{\widehat{R}/A} \widehat{\otimes}_{\widehat{R}} S \to \widehat{\Omega}_{S/A} \to \widehat{\Omega}_{S/\widehat{R}} \to 0.
\end{equation}
Since $\widehat{R} \to S$ is module-finite, we know that $\widehat{\Omega}_{\widehat{R}/A}=\Omega_{\widehat{R}/A}$ and that $\widehat{\Omega}_{\widehat{R}/A}$ is spanned by the image of $dx_0,\ldots,dx_d$ as an $\widehat{R}$-module. From the presentation $\widehat{\overline{R}} \cong \widehat{R} \otimes_R \overline{R}$, it follows that $\widehat{\overline{R}}$ is the integral closure of $\widehat{R}$ in the total ring of fractions. In particular, the map $\widehat{R} \to S$ is unramified in codimension $1$. We also note that if $\fp \in \Spec \widehat{R}$ for which $\widehat{R}_{\fp} \to S_\fp$ is unramified, then $\widehat{\Omega}_{S/\widehat{R}}$ vanishes after localizing at $\fp$. Set
$$
Q_1:=\Big\{\fp \in \Spec^\circ(\widehat{R})~\Big|~\fp \  \mbox{is a minimal prime of}~\Supp(\widehat{\Omega}_{S/\widehat{R}})\Big\},
$$
which is a finite set. As $\widehat{R} \to S$ is unramified in codimension $1$, every $\fp \in Q_1$ has height at least $2$. Set $\Spec^2(\widehat{R})$ to be the set of all prime ideals of height $2$. For all $\fp \in \Spec^2(\widehat{R}) \setminus Q_1$ (in other words, almost all prime ideals in $ \Spec^2(\widehat{R})$), we get the vanishing: $(\widehat{\Omega}_{S/\widehat{R}} )_\fp=0$, which shows that the $S_\fp$-module $(\widehat{\Omega}_{S/A})_\fp$ is spanned by the images of $dx_0,\ldots,dx_d$ as can be deduced from the exact sequence $(\ref{KahlerExt})$.

Recall that $A$ is a coefficient ring for $\widehat{R}$, but may not be for $S$, because of the possibility of the non-triviality of the residue field extension for $\widehat{R} \to S$. By Lemma \ref{coefficientring}, one can construct a coefficient ring map $B \to S$ and a commutative square:
$$
\begin{CD}
\widehat{R} @>>> S \\
@AAA @AAA \\
A @>>> B\\
\end{CD}
$$

As the natural map $\widehat{\Omega}_{S/A} \to \widehat{\Omega}_{S/B}$ is surjective, we get the following:
\begin{equation}
\label{Kahlermap}
\mbox{The natural map}~\bigoplus_{i=0}^m S_\fp dx_i \to (\widehat{\Omega}_{S/B})_\fp~\mbox{is surjective for every}~\fp \in \Spec^2(\widehat{R}) \setminus Q_1.
\end{equation}
Applying Theorem \ref{local-Bertini-Normal} in conjunction with $(\ref{Kahlermap})$ yields the following: There exists a non-empty open subset $\mathcal{U}' \subseteq \mathbb{P}^{d}(K)$ such that
\begin{equation}
\label{secondsymbolic}
\mathbf{x}_{\widetilde{\alpha}} := \sum_{i=0}^d \widetilde{\alpha}_i x_i \notin \fp^{(2)}
\end{equation}
for every $\fp \in \Spec^2(S) \setminus \widetilde{Q}_1$ and for every $\alpha=(\alpha_0:\cdots:\alpha_d) \in \Sp^{-1}_B(\mathcal{U}') \subseteq \mathbb{P}^{d}(B)$. Here $\widetilde{Q}_1$ is the inverse image of $Q_1$ under the finite map $\Spec(S) \to \Spec(\widehat{R})$, so that it is again finite. To get $\mathbf{x}_{\widetilde{\alpha}} \notin \fp$ for every $\fp \in \widetilde{Q}_1$, we apply \cite[Lemma 4.2]{OS15} to each prime ideal in $\widetilde{Q}_1$ and we have a non-empty open subset $\mathcal{U}'' \subset \mathbb{P}^{d}(K)$. For $\alpha=(\alpha_0: \cdots :\alpha_d) \in \Sp^{-1}_B(\mathcal{U}' \cap \mathcal{U}'')$, let $\mathbf {x}_{\widetilde{\alpha}}=\sum_{i=0}^d \widetilde{\alpha}_i x_i$. Then if $\fp \in \Spec^2(S) \cap V(\mathbf{x}_{\widetilde{\alpha}})$, since $\mathbf {x}_{\widetilde{\alpha}}$ is not contained in any prime ideal of $\widetilde{Q}_1$, the localization $S_{\fp}$ is regular. Therefore, $(S/\mathbf{x}_{\widetilde{\alpha}}S)_\fp$ is a discrete valuation ring in view of $(\ref{secondsymbolic})$.

Next we examine the $(S_2)$-condition, which is to say that $\Ht(\fp)>2$. Since $S$ is a complete local normal domain,
$$
Q_2=\Big\{\fp \in \Spec^\circ(S)~\Big|~\depth {S_{\fp}}=2~\mbox{and}~\dim{S_{\fp} > 2} \Big\}
$$ 
is a finite set by \cite[Lemma 3.2]{Fl77}. Again applying \cite[Lemma 4.2]{OS15} to each prime ideal in $Q_2$, we have an open subset $\mathcal{U}''' \subset \mathbb{P}^d(K)$ such that $\depth (S/\mathbf{x}_{\widetilde{\alpha}}S)_\fp \ge 2$, where $\alpha=(\alpha_0:\cdots:\alpha_d) \in \Sp^{-1}_B(\mathcal{U}''')$ and $\fp \in \Spec^\circ(S) \cap V(\mathbf{x}_{\widetilde{\alpha}})$ has height at least $3$. Now there is a commutative diagram of projective spaces:
$$
\begin{CD}
\mathbb{P}^d(V) @>>> \mathbb{P}^d(A) @>>> \mathbb{P}^d(B) \\
@V\Sp_{V}VV @V\Sp_AVV @V\Sp_BVV \\
\mathbb{P}^d(k) @= \mathbb{P}^d(k) @>>> \mathbb{P}^d(K) \\
\end{CD}
$$
Let $\mathcal{V} \subset \mathbb{P}^d(k)$ be the inverse image of $\mathcal{U}' \cap \mathcal{U}'' \cap \mathcal{U}''' \subset \mathbb{P}^d(K)$. Then \cite[Proposition 2.5]{OS15} and the commutativity of the above diagram allow us to conclude that $\mathcal{V}$ is the desired non-empty open set establishing $(\#)$.

Let us now return to the notation as in $(\#)$. Let $K_i$ be the residue field of $S_i$. We have the canonical mapping:
$$
\mathbb{P}^d(k) \to \prod_{i=1}^m \mathbb{P}^d(K_i)
$$
and let $\mathcal{U} \subset \mathbb{P}^d(k)$ be the inverse image of the subset $\prod_{i=1}^m \mathcal{V}_i \subset \prod_{i=1}^m \mathbb{P}^d(K_i)$. Again by \cite[Proposition 2.5]{OS15}, $\mathcal{U}$ is a non-empty open subset of $\mathbb{P}^d(k)$, and the following holds: If $\mathbf{x}_{\widetilde{\alpha}}:=\sum_{i=0}^d \widetilde{\alpha}_i x_i$ for $\alpha=(\alpha_0:\cdots:\alpha_d) \in \Sp^{-1}_V(\mathcal{U})$, then $(\widehat{\overline{R}}/\mathbf{x}_{\widetilde{\alpha}}\widehat{\overline{R}})_\fp$ is a normal domain for $\fp \in \Spec^\circ(\widehat{\overline{R}}) \cap V(\mathbf{x}_{\widehat{\alpha}})$, and the element $\mathbf{x}_{\widetilde{\alpha}}$ belongs to $\overline{R}$. The map induced by $(\ref{faithfulflat})$
$$
\overline{R}/\mathbf{x}_{\widetilde{\alpha}}\overline{R} \to \widehat{\overline{R}}/\mathbf{x}_{\widetilde{\alpha}}\widehat{\overline{R}}
$$
is faithfully flat, thus implying the normality of $(\overline{R}/\mathbf{x}_{\widetilde{\alpha}}\overline{R})_\fp$.  Finally, if $\depth_{\fm} \overline{R} \ge 3$ holds, then $\depth_\fm \overline{R}/\mathbf{x}_{\widetilde{\alpha}}\overline{R} \ge 2$ and $\overline{R}/\mathbf{x}_{\widetilde{\alpha}}\overline{R}$ is normal. To make $\overline{R}/\mathbf{x}_{\widetilde{\alpha}}\overline{R}$ of mixed characteristic, one applies \cite[Proposition 2.5]{OS15} to the finite set of height-$1$ primes of $\overline{R}$
 containing $p$, concluding the proof.
\end{proof}

Now we have the following result.

\begin{theorem}
\label{weaknormalBertini}
Let $(V,\pi,k)$ be an unramified complete discrete valuation ring of mixed characteristic with residue field of characteristic $p>0$. Suppose that $(R,\fm,k)$ is a $V$-flat complete local domain which induces an isomorphism on residue fields $k \cong V/(\pi) \cong R/\fm$ and that the following conditions hold:
\begin{enumerate}
\item
$R \to \overline{R}$ is unramified in codimension $1$, where $\overline{R}$ is the integral closure of $R$ in the field of fractions of $R$;

\item
$x_0,\ldots,x_d$ is a fixed set of minimal generators of $\fm$;

\item
the residue field $k$ is infinite.
\end{enumerate}

Then there exists a non-empty Zariski open subset $\mathcal{U} \subset \mathbb{P}^{d}(k)$ such that for a fixed $\alpha=(\alpha_0:\cdots:\alpha_d) \in \Sp^{-1}_V(\mathcal{U})$, the localization $(R/\mathbf{x}_{\widetilde{\alpha}}R)_\fp$ is a weakly normal local reduced ring of mixed characteristic for every $\fp \in \WN(R) \cap V(\mathbf{x}_{\widetilde{\alpha}}) \cap \Spec^\circ(R)$, where we put
$$
\mathbf{x}_{\widetilde{\alpha}}=\sum_{i=0}^d \widetilde{\alpha}_i x_i.
$$
\end{theorem}

\begin{proof}
Denote by $\overline{R}$ the integral closure of $R$ in its field of fractions and
consider the complex of $R$-modules: 
\begin{equation}
\label{complexexact}
C_\bullet:R \xrightarrow{\phi_1} \overline{R} \xrightarrow{\phi_2} (\overline{R} \otimes_R\overline{R})_{\rm{red}}
\end{equation}
as defined in Theorem \ref{characterization} with $\phi_2:=f-g$. Recall that the localization $(C_\bullet)_\fp$ is exact for $\fp \in \WN(R)$. By Proposition \ref{semilocalnormalBertini}, there is a non-empty Zariski open subset $\mathcal{U}' \subset \mathbb{P}^d(k)$ such that the localization of $\overline{R}/\mathbf{x}_{\widetilde{\alpha}}\overline{R}$ is normal at every $\fp \in \Spec^\circ(R) \cap V(\mathbf{x}_{\widetilde{\alpha}})$ and every $\alpha=(\alpha_0:\cdots:\alpha_d) \in \Sp^{-1}_V(\mathcal{U}')$. Let
$$
I:=\big\{\fp~\big|~\fp \in \Ass(\coker(\phi_1)) \cup \Ass(\coker(\phi_2))\big\}.
$$
Next we need to shrink $\mathcal{U}'$ to some appropriate open set in order to force $\mathbf{x}_{\widetilde{\alpha}} \notin \fp$ for every $\fp \in I$. To this aim, we define an open subset $\mathcal{U}_\fp \subset \mathbb{P}^d(k)$ for each $\fp \in I$ by applying \cite[Lemma 4.2]{OS15}. Set $\mathcal{U}:=\mathcal{U}' \cap\big(\cap_{\fp \in I} \mathcal{U}_\fp\big)$. Then by applying both \cite[Lemma 1.1]{CGM83} and \cite[Lemma 1.6]{CGM83}, we find that every element $\mathbf{x}_{\widetilde{\alpha}}$ attached to $\mathcal{U}$ satisfies the following property:
\begin{enumerate}
\item[$(*)$]
The natural map $R/\mathbf{x}_{\widetilde{\alpha}}R \to \overline{R}/\mathbf{x}_{\widetilde{\alpha}}\overline{R}$ is injective and torsion free. Whenever the localization $(C_\bullet)_\fp$ of the complex as in $(\ref{complexexact})$ is exact, then the induced complex of
$R/\mathbf{x}_{\widetilde{\alpha}}R$-modules $(C_\bullet/\mathbf{x}_{\widetilde{\alpha}}C_\bullet)_\fp$ stays exact.
\end{enumerate}

In other words, $(*)$ asserts the following: If $\mathbf{x}_{\widetilde{\alpha}}$ is attached to $\mathcal{U}$ and $\fp \in WN(R) \cap V(\mathbf{x}_{\widetilde{\alpha}}) \cap \Spec^\circ(R)$, then $(\overline{R}/\mathbf{x}_{\widetilde{\alpha}} \overline{R})_{\fp}$ is a semi-local reduced normal ring, the complex induced by $(\ref{complexexact})$:
$$
(R/\mathbf{x}_{\widetilde{\alpha}}R)_{\fp} 
\to (\overline{R}/\mathbf{x}_{\widetilde{\alpha}} \overline{R})_{\fp}
\to \big((\overline{R}/\mathbf{x}_{\widetilde{\alpha}} \overline{R})_{\fp}
\otimes_{(R/\mathbf{x}_{\widetilde{\alpha}} R)_{\fp}} (\overline{R}/\mathbf{x}_{\widetilde{\alpha}} \overline{R})_{\fp}\big)_{\rm{red}}
$$
is exact, and $R/\mathbf{x}_{\widetilde{\alpha}}R
\to \overline{R}/\mathbf{x}_{\widetilde{\alpha}} \overline{R}$ is the normalization map. Hence Theorem \ref{characterization} applies and allows us to conclude that $(R/\mathbf{x}_{\widetilde{\alpha}}R)_{\fp}$ is weakly normal. This completes the proof of the theorem.
\end{proof}

We obtain the local Bertini theorem for weak normality as a corollary.

\begin{corollary}
In addition to the notation and hypotheses of Theorem \ref{weaknormalBertini}, suppose that the complete local domain $R$ is weakly normal. Then there exists a non-empty Zariski open subset $\mathcal{U} \subset \mathbb{P}^{d}(k)$ such that for every $\alpha=(\alpha_0:\cdots:\alpha_d) \in \Sp^{-1}_V(\mathcal{U})$, the following holds:
\begin{enumerate}
\item
$(R/\mathbf{x}_{\widetilde{\alpha}}R)_\fp$ is weakly normal for every $\fp \in V(\mathbf{x}_{\widetilde{\alpha}}) \cap \Spec^\circ(R)$.

\item 
If moreover $\depth R \ge 3$, then $R/\mathbf{x}_{\widetilde{\alpha}}R$ is weakly normal.
\end{enumerate}
\end{corollary}

\begin{proof}
The statement $(1)$ immediately follows from Theorem \ref{weaknormalBertini}, while the statement $(2)$ follows from \cite[Corollary (IV. 4)]{M80}.
\end{proof}

We considered a tantalizing condition that the normalization map $R \to \overline{R}$ is unramified in codimension $1$. This is satisfied, for example, when $R$ has Serre's $(R_1)$-condition. One should notice that weakly normal local rings possessing $(R_1)$-condition are not always normal. Indeed, the authors of \cite{ES16} introduced a certain class of commutative (not necessarily Noetherian) rings, called \textit{perinormal rings}. According to \cite[Proposition 3.2 and Corollary 3.4]{ES16}, any Noetherian perinormal ring is weakly normal satisfying $(R_1)$-condition. Based on \cite[Example 3.6]{ES16}, we present an example of a weakly normal complete local domain of mixed characteristic, which is not normal, but which possesses $(R_1)$-condition with infinite residue class field.

\begin{example}
\label{weaklynormalR_1}
Let $k$ be an algebraically closed field of characteristic $p>0$ and let $W(k)$ be the ring of Witt vectors. Assume that the characteristic of $k$ is different from $2$. Let us consider the subring:
$$
R:=W(k)[[X,Y,XZ,YZ,Z^2]] \subset W(k)[[X,Y,Z]].
$$
Then we can check this example satisfies all the desired conditions. First off, the normalization of $R$ is $W(k)[[X,Y,Z]]$ and hence, $R$ is not normal. As $Z(X,Y) \subset R$ and $(X,Y) \subset R$ is a height-$2$ ideal, $R$ satisfies $(R_1)$-condition (here, it suffices to recall that $R$ satisfies $(R_1)$-condition if and only if the conductor of the normalization map $R \to W(k)[[X,Y,Z]]$ is not contained in any height-$1$ prime ideal). Since $R/pR$ is weakly normal, one applies \cite[Corollary 4.1 in the excellent case]{BF13} or \cite[Proposition 4.11 in the general case]{Mu20} to conclude that $R$ is also weakly normal.
\end{example}

\begin{question}
We pose some questions.
\begin{itemize}
\item Can we remove the assumption that $(V,\pi,k)$ is unramified or $R \to \overline{R}$ is unramified in codimension $1$?
\item Can we formulate and prove the local Bertini theorem for the case when the residue class field is finite? 
\item Can we prove the local Bertini theorem for seminormality?
\end{itemize}
\end{question}

At this point, it is necessary to assume that $R \to \overline{R}$ is unramified in codimension $1$. This was previously studied in the paper \cite{CM81} 
as "$\WN1$-condition" ($=$seminormal+unramified in codimension $1$ for the normalization map), and then the authors of \cite{CGM83} used it to establish the global Bertini theorem for weak normality over an arbitrary algebraically closed field. On the other hand, the global Bertini theorem for weak normality in positive characteristic without $\WN1$-condition was refuted in \cite{CGM89}. This seems to suggest to us that an appropriate formulation for the local Bertini theorem for weak normality in positive characteristic requires the $\WN1$-condition in an essential way. In the finite residue field case, there is some recent work for the Bertini theorem over projective schemes; see the paper \cite{GK19}.

\begin{acknowledgement}
The authors would like to thank Prof. V. Trivedi for discussing local Bertini theorems. The first author wants to express his gratitude to Prof. N. Epstein for discussions on examples of weakly normal rings. The authors would like to thank Prof. K. Schwede for discussions on pull-backs. Our gratitude also goes to Dr. T. Murayama for a remark on the definition of seminormality. Finally, the authors thank the referee for going well beyond the due jobs and having provided various comments and constructive suggestions patiently. The second author was partially supported by JSPS Grant-in-Aid for Scientific Research(C) 18K03257.
\end{acknowledgement}

\end{document}